\documentclass{amsart}

\usepackage{amssymb}
\usepackage[T1]{fontenc}

\newtheorem{teorema}{Theorem}

\newtheorem{definicija}{Definition}

\DeclareMathOperator*{\Limsup}{\overline{lim}}
\DeclareMathOperator*{\Liminf}{\underline{lim}}

\begin{document}

\title{On the continuity of probabilistic distance}

\author{Dragoljub J.\ Ke\v{c}ki\'c}

\address{University of Belgrade \\ Faculty of Mathematics \\ Studentski trg 16 \\ 11000 Beo\-grad \\ Serbia}

\email{keckic@matf.bg.ac.rs}

\author{Marina M.\ Milovanovi\'c-Aran{\dj}elovi\'c}

\address{Technical College of Applied Studies \\ Tehnikum Taurunum \\ Nade Dimi\'c 4 \\ 11080 Zemun \\ Serbia}
\email{mmarandjelovic@tehnikum.edu.rs}

\begin{abstract}
The famous result of B.~Schweizer and A.~Sklar [Pacific J Math 10(1960) 313--334 - Theorem 8.2] asserts that, given a probabilistic metric space $(X,\mathcal F,t)$, $\mathcal F=\{F_{p,q}:p,q\in X\}$, we have $F_{p_n,q_n}(x)\to F_{p,q}(x)$ provided that $F_{p,q}$ is continuous at $x$ and $t$ is continuous and stronger then {\L}ukasiwicz's $t$-norm. We extend this result to arbitrary continuous triangular norms, i.e.\ we omit the condition "$t$ is stronger then {\L}ukasiewicz's".
\end{abstract}



\maketitle

\section{Introduction and Preliminaries}

Probabilistic metric spaces were introduced by K.~Menger \cite{KM}. In these spaces, instead of the distance $d(p,q)$ between points $p$ and $q$, we consider the distribution functions $F_{p,q}(x)$ and interpret its values as the probability that the distance from $p$ to $q$ is less than $x$.

Recall that a distribution function is a nondecreasing, left continuous mapping $F:\mathbb R\to[0,1]$ such that $\inf \{F(x):x\in{\mathbb R}\}=0$ and $\sup \{F(x):x\in{\mathbb R}\}=1$. Let $\mathcal L$ denote the collection of all distribution functions $F:{\mathbb R}\rightarrow [0,1]$.

\begin{definicija} A probabilistic metric space is an ordered pair $(X,\mathcal{F})$ where $X$ is an nonempty set and $\mathcal{F}$ is a mapping from $X \times X$ to ${\mathcal L}$. The value of $\mathcal{F}$ at $(p,q) \in X \times X$ will be denoted by $F_{p,q}$. The functions $F_{p,q}$ ($p$, $q \in X$) are assumed to satisfy the following conditions:

\begin{enumerate}

\item[(a)] $F_{p,q}(x)=1$ for all $x > 0$ if and only if $p=q$;

\item[(b)] $F_{p,q} (0)=0$;

\item[(c)] $F_{p,q} =F_{q,p}$

\item[(d)] $F_{p,q}(x)=1$ and $F_{q,r}(y)=1$ imply $F_{p,r}(x+y)=1$.
\end{enumerate}
\end{definicija}

Given a probabilistic metric space $(X,{\mathcal F})$, the family of sets
$$U_p(\varepsilon,\lambda)=\{q \in X|F_{p,q}(\varepsilon)>1-\lambda\},\qquad p\in X,~\varepsilon>0,~\lambda>0$$
make a basis of some topology on $X$. The set $U_p(\varepsilon,\lambda)$ is called  $(\varepsilon, \lambda)$--neighborhood of $p \in X$. In this topology, a sequence $(p_n)$ converges to $p$ if and only if for each $\varepsilon > 0$ and each $\lambda > 0$, there exists positive integer $M_{\varepsilon,\lambda}$, such that $p_{n} \in U_{p}(\varepsilon,\lambda)$, i.e.\ $F_{p,p_{n}}(\varepsilon) > 1 - \lambda$ for all $n \geq M_{\varepsilon, \lambda}$.

For more details on the topological preliminaries the reader is referred to \cite{SS}.

\begin{definicija}\label{TrNorm} A mapping $t:[0,1] \times [0,1] \rightarrow [0,1]$ is a triangular norm ($t$-norm) if it satisfies:

\begin{enumerate}
\item\label{TrNorm1} $t(a,1)=a,$ $t(0,0)=0$,

\item\label{TrNorm2} $t(a,b)=t(b,a)$,

\item\label{TrNorm3} $t(c,d) \geq t(a,b)$ for $c \geq a$, $d \geq b$,

\item\label{TrNorm4} $t(t(a,b),c)=t(a,t(b,c))$.
\end{enumerate}
\end{definicija}

An example of $t$-norm is \emph{{\L}ukasiewicz's triangular norm}, defined by $t(a,b)=\max\{a+b-1,0\}$ (see \cite{KMP}).

\begin{definicija} Let $t_1$ and $t_2$ be two triangular norms. We say that $t_2$ is stronger then $t_1$ if and only if:
\begin{enumerate}
\item $t_2(a,b)\ge t_1(a,b)$ for all $a,b\in [0,1]$;

\item there exists $c,d\in [0,1]$ such that $t_2(c,d)> t_1(c,d)$.
\end{enumerate}
\end{definicija}

\begin{definicija}A Menger space is a triplet $(X,{\mathcal F},t)$, where $(X,{\mathcal F})$ is a probabilistic metric space and $t$ is a triangular norm which satisfies the Menger's triangle inequality
$$F_{p,r} (x+y) \geq t[F_{p,q}(x),F _{q,r}(y)]$$
for all $p$, $q$, $r\in X$ and for all $x \geq 0$, $y \geq 0$.
\end{definicija}

Next two theorems were presented by B.~Schweizer and A.~Sklar \cite{SS}.

\begin{teorema} [B.~Schweizer and A.~Sklar \cite{SS} - Theorem 8.1]\label{SS1} If $(X,{\mathcal F},t)$  is a Menger space and $t$ is continuous, then ${\mathcal F}$ is a lower semi-continuous function of points, i.e., for every fixed x, if $q_n\to q$ and $p_n\to p$, then
$$\Liminf_{n\rightarrow\infty}F_{p_{n},q_{n}}(x)=F_{p,q}(x).$$
\end{teorema}

\begin{teorema} [B.~Schweizer and A.~Sklar \cite{SS} - Theorem 8.2]\label{SS2} Let $(X,{\mathcal F},t)$ be a Menger space. Suppose that $t$ is continuous and at least as strong as {\L}ukasziewicz norm. Suppose further
that $p_n\to p$, $q_n\to q$, and that $F_{p,q}$ is continuous at $x$. Then
$$\lim_{n\to\infty}F_{p_n,q_n}(x)=F_{p,q}(x).$$
\end{teorema}

However, there are continuous $t$-norms weaker then {\L}ukasiewicz. For instance all Sweizer-Sclar norms $t_p(a,b)=\max\{(a^p+b^p-1)^{1/p},0\}$ for $1<p<\infty$. Therefore, the condition "$t$ is stronger then {\L}ukasiewicz's triangular norm" in the statement of Theorem \ref{SS2} is restrictive. In this paper we shall prove that it can be omitted.

\section{Results}

In this paper our main result is the following theorem, which extend Theorem 2 to Menger spaces with arbitrary continuous triangular norm.

\begin{teorema} Let $(X,{\mathcal F},t)$ be a Menger space, where $t$ is continuous triangular norm, $p,q\in X$ and $(p_{n}),(q_{n})\subseteq X$ such that $p_{n}\rightarrow p$, $q_{n} \rightarrow q$. If $F_{p,q}$ is continuous at $x\in X$, then
$$\lim_{n\rightarrow\infty} F_{p_{n},q_{n}}(x)=F_{p,q}(x).$$
\end{teorema}

\begin{proof} Let $t_*(a,b,c)=t(t(a,b),c)$. By Definition \ref{TrNorm}.--(\ref{TrNorm2}) and (\ref{TrNorm4}), the value of $t_*$ does not depend on permutation of numbers $a$, $b$ and $c$. Also $t_*$ is increasing on coordinates, and continuous, because $t$ is continuous.

First, let us prove that for arbitrary $a_n$, $b_n$, $c_n\in[0,1]$, $a_n$, $c_n\to1$ implies
\begin{equation}\label{GornjiLimes}
\Limsup_{n\to\infty}t_*(a_n,b_n,c_n)=\Limsup_{n\to+\infty}b_n.
\end{equation}

Indeed, denote $\beta=\Limsup\limits_{n\to\infty}b_n$. Then for any $\varepsilon >0$, there is $n_0$ such that from $n\ge n_0$ it follows $b_n<\beta +\varepsilon$. By monotonicity of $t_*$ and Definition \ref{TrNorm}.--(\ref{TrNorm1}) we obtain
$$t_*(a_n,b_n,c_n)\le t_*(1,\beta +\varepsilon,1)=\beta +\varepsilon,$$
which implies:
\begin{equation}\label{GornjiLimes1}
\Limsup_{n\to\infty}t_*(a_n,b_n,c_n)\le\beta.
\end{equation}
On the other hand, there is a subsequence $\{b_{n_k}\}\subseteq \{b_n\}$ such that $\lim\limits_{k\to\infty}b_{n_k}=\beta$, and hence
\begin{equation}\label{GornjiLimes2}
\lim_{k\rightarrow\infty}t_*(a_{n_k},b_{n_k},c_{n_k})=t_*(1,\beta,1)=\beta,
\end{equation}
because $t_*$ is a continuous function. From (\ref{GornjiLimes1}) and (\ref{GornjiLimes2}) we get (\ref{GornjiLimes}).

Next, for all $p$, $q$, $r$, $s\in X$ and $a$, $b$, $c\in [0,1]$, by Menger's triangle inequality, and monotonicity of $t$, we have
\begin{equation}\label{Prob}
\begin{aligned}F_{p,s}(a+b+c)&\ge t(F_{p,r}(a+b),F_{r,s}(c))\ge\\
    &\ge t(t(F_{p,q}(a),F_{q,r}(b)),F_{r,s}(c))=\\
    &=t_*(F_{p,q}(a),F_{q,r}(b),F_{r,s}(c)).
\end{aligned}
\end{equation}

Now, let $p_n\to p$ and $q_n\to q$, and let $F_{p,q}$ be continuous at $x$. We have
\begin{equation}\label{PQ}
\lim\limits_{n\to\infty}F_{p_n,p}(\varepsilon)=1,\qquad\lim\limits_{n\to\infty}F_{q_n,q}(\varepsilon)=1.
\end{equation}
By (\ref{Prob}) for all $\varepsilon >0$, we have
$$F_{p,q}(x+2\varepsilon)\ge t_*(F_{p,p_n}(\varepsilon),F_{p_n,q_n}(x),F_{q_n,q}(\varepsilon)).$$
Passing to $\Limsup$ and taking account (\ref{PQ}) and (\ref{GornjiLimes}), we obtain
$$F_{p,q}(x+2\varepsilon)\ge \Limsup_{n\to\infty}t_*(F_{p,p_n}(\varepsilon),F_{p_n,q_n}(x),F_{q_n,q}(\varepsilon))
=\Limsup_{n\to\infty}F_{p_n,q_n}(x).$$
By continuity of $F_{p,q}$ (since $\varepsilon>0$ was arbitrary)
$$F_{p,q}(x)\ge\Limsup_{n\to\infty}F_{p_n,q_n}(x).$$

Finally, from Theorem \ref{SS1}, we get
$$\Liminf_{n\to\infty}F_{p_{n},q_{n}}(x)=F_{p,q}(x)\ge \Limsup_{n\to\infty}F_{p_n,q_n}(x),$$
which is equivalent to
$$\lim_{n\rightarrow\infty}F_{p_{n},q_{n}}(x) = F_{p,q}(x).$$
\end{proof}

\bibliographystyle{abbrv}
\bibliography{MengerBibl}

\begin{thebibliography}{1}

\bibitem{KMP}
E.~P. Klement, R.~Mesiar, and E.~Pap.
\newblock {\em Triangular norms}, volume~8 of {\em Trends in Logic}.
\newblock Kluwer Akademic Publishers, Dordrecht, 2000.

\bibitem{KM}
K.~Menger.
\newblock Statistical metrics.
\newblock {\em Proc. Nat. Acad. Sci. U.S.A.}, 28(12):535--537, 1942.

\bibitem{SS}
B.~Schweizer and A.~Sklar.
\newblock Statistical metric spaces.
\newblock {\em Pacific J Math.}, 10(1):313--334, 1960.

\end{thebibliography}

\end{document}